\documentclass[a4paper]{amsart}
\usepackage{etex}
\usepackage[utf8]{inputenc}
\usepackage[T1]{fontenc}
\usepackage{lmodern}
\usepackage[english]{babel}
\usepackage{textcomp}
\usepackage{enumitem}
\usepackage{microtype}
\usepackage{graphicx}

\calclayout

\usepackage{amsmath}						
\usepackage{amssymb}						
\usepackage{amsthm}
\usepackage{amstext}	

\usepackage{bbm}
\usepackage{xfrac}

\theoremstyle{plain}
\newtheorem{lemma}{Lemma}
\newtheorem{theorem}[lemma]{Theorem}

\theoremstyle{definition}
\newtheorem{definition}[lemma]{Definition}
\theoremstyle{remark}
\newtheorem*{remark}{Remark}

\DeclareMathOperator{\rad}{rad}
\DeclareMathOperator{\sqf}{sqf}
\DeclareMathOperator{\re}{Re}
\DeclareMathOperator{\res}{res}

\newcommand{\ssum}{\sideset{}{^*}\sum}

\begin{document}

\title{Counting multi-quadratic number fields of bounded discriminant}
\author{Robin Fritsch}
\address{Technical University of Munich}
\email{robin.fritsch@tum.de}
\date{}

\begin{abstract}
We prove an asymptotic formula for the number of multi-quadratic number fields of bounded discriminant with a power-saving error term. Furthermore, we explicitly calculate the leading coefficient and extend our result to totally real multi-quadratic number fields.
\end{abstract}

\maketitle
 
\section{Introduction}
We consider a special case of the general problem of counting number fields of bounded discriminant. An overview of results regarding this problem is given in the surveys \cite{cohen02} and \cite{cohen06}. Using class field theory, Wright \cite{wright89} derived an asymptotic formula for the number of abelian extensions of a global field with fixed Galois group and bounded discriminant. In particular, the number $N_k(x)$ of number fields with Galois group $(\mathbb{Z}/2)^k$, which are exactly the multi-quadratic number fields, and discriminant $D(K)\leq x$ can thereby be evaluated asymptotically as
$N_k(x) \sim C(k) x^{1/2^{k-1}}(\log x)^{2^k-2}$
with a constant $C(k)\in\mathbb{R}$.

We use elementary methods to prove an asymptotic formula for $N_k(x)$ with a power-saving error term and explicitly calculate the leading coefficient.
Recently, this result has also been published by de la Bretèche et al. \cite{breteche19} generalizing work by Rome \cite{rome17} on biquadratic extensions. Our work was done independently as a bachelor's thesis\footnote{\textbf{Acknowledgement:} I would like to thank Valentin Blomer for
suggesting the topic and for helpful discussion and advice.}  at the University of Göttingen in 2016.
Additionally, we derive the corresponding result when we consider only totally real multi-quadratic number fields.

Let $N_k(x)$ be the number of number fields $K=\mathbb{Q}(\sqrt{a_1},\ldots,\sqrt{a_k})$ with $[K:\mathbb{Q}]=2^k$ and bounded discriminant $D(K)\leq x$. Accordingly, let $N_k^+(x)$ be the number of totally real number fields. The discriminant of such number fields essentially depends on the product $a_1\ldots a_k$ (see Schmal \cite{schmal89}). We use basic combinatorics to determine the number of fields when this product is fixed. More precisely, we count the number of tuples $(a_1,\ldots,a_k)$ which we call \emph{presentation} of a number field. In order not to count different presentations of the same number field multiple times, we define a unique \emph{normal} presentation. From this we can derive the number of fields with bounded discriminant and obtain our main result.
\begin{theorem}\label{theorem:main_result}
For $k\geq 2$ there exists a $d<1$ as well as two monic polynomials $P_k$ and $P_k^+$ of degree $2^k-2$ such that
\begin{equation*}
N_k(x) = C_k x^{1/2^{k-1}} P_k(\log x) + \mathcal{O}\left(x^{d/2^{k-1}}\right)
\end{equation*}
and
\begin{equation*}
N_k^+(x) = \frac{1}{2^k} C_k x^{1/2^{k-1}} P_k^+(\log x) + \mathcal{O}\left(x^{d/2^{k-1}}\right)
\end{equation*}
where
\begin{equation*}
C_k = \frac{4^k+5\cdot 2^k+10}{32(2^k-1)!} \left(\frac{1}{2^k}\right)^{2^k-2} \prod_{j=1}^{k-1} \frac{1}{2^k-2^j} \prod_{p\neq 2}\left(1+\frac{2^k-1}{p}\right) \left(1-\frac{1}{p}\right)^{2^k-1}.
\end{equation*}
\end{theorem}

\section{Counting number fields}
In the following we consider number fields $K=\mathbb{Q}(\sqrt{a_1},\ldots,\sqrt{a_k})$ for $k\geq 2$ where $[K:\mathbb{Q}]=2^k$ and $a_1,\ldots,a_k$ are square-free integers. We refer to such a number field simply as a multi-quadratic number field and call the tuple $(a_1,\ldots,a_n)$ its presentation. Without loss of generality we may assume that
\begin{equation}
\begin{split}
(a_1,a_2) &\equiv (1,1),(2,1),(3,1) \text{ or } (2,3) \mod 4\\
a_i &\equiv 1 \mod 4 \qquad\text{ for } i\geq 3
\end{split} \label{eq:mod4}
\end{equation}
holds (see Schmal \cite{schmal89}). Furthermore, Schmal \cite{schmal89} provides a formula for the discriminant $D(K)$ of such a number field.

\begin{theorem}\label{theorem:discriminant}
Let $K=\mathbb{Q}(\sqrt{a_1},\ldots,\sqrt{a_k})$ be a number field satisfying \eqref{eq:mod4}. Then
\begin{equation*}
D(K)=(2^r \rad(a_1\ldots a_k))^{2^{k-1}}
\end{equation*}
where
\begin{equation*}
r=
\begin{cases}
0 &\text{ for } (a_1,a_2)\equiv (1,1) \mod 4 \\
2 &\text{ for } (a_1,a_2)\equiv (2,1) \text{ or } (3,1) \mod 4 \\
3 &\text{ for } (a_1,a_2)\equiv (2,3) \mod 4.
\end{cases}
\end{equation*}
\end{theorem}

Since the discriminant mainly depends on $a_1\ldots a_k$, we will determine the number of number fields when this product is fixed. More precisely, we will count their representations. In order to avoid counting different presentations of the same number field, we define a unique \emph{normal} presentation.

\begin{definition}
Let $\sqf(n)$ and $\rad(n)$ be the square-free part and the radical of $n\in\mathbb{Z}$, respectively, i.e. if $n=\pm p_1^{e_1}\cdots p_n^{e_n}$, then $\rad(n)=p_1\cdots p_n$ and $\sqf(n)=\pm p_1^{f_1}\cdots p_n^{f_n}$ where $f_i\in\{0,1\}$ and $f_i\equiv e_i$ (mod 2) for $i=1,\ldots,n$.
\end{definition}

\begin{definition}
We call a multi-quadratic number field $i$-free if no set $S\subseteq \{1,\ldots,k\}$ exists such that $\sqf(\prod_{j\in S} a_j)=-1$. In other words, $K$ is $i$-free if and only if $i\not\in K$.
\end{definition}

\begin{definition}\label{def:normal}
Let $K=\mathbb{Q}(\sqrt{a_1},\ldots,\sqrt{a_k})$ be an $i$-free multi-quadratic number field. Furthermore, let $p_1,\ldots,p_n$ be primes such that $\rad(a_1\ldots a_k) = p_1\ldots p_n$.\\
For $j=1,\ldots,k$ let $i_j:=\min\{i\in\{1,\ldots,n\}:p_i\mid a_j\}$. We say the presentation $(a_1,\ldots,a_n)$ is \emph{normal} if
\begin{equation}
i_1<i_2<\ldots<i_k \text{ and } i_1=1 \tag{a}\label{eq:normal_a}
\end{equation}
as well as
\begin{equation}
p_{i_j}\nmid a_i \quad\forall i\neq j \quad\forall j\in\{1,\ldots,k\}.\tag{b}\label{eq:normal_b}
\end{equation}
\end{definition}

In the following $i_1,\ldots,i_k$ and $p_1,\ldots,p_n$ will always be used as in the definition above. Furthermore let $P=p_1\ldots p_n$.

\begin{lemma}
Every $i$-free multi-quadratic number field has a normal presentation.
\end{lemma}
\begin{proof}
Let $K=\mathbb{Q}(\sqrt{a_1},\ldots,\sqrt{a_k})$. We will transform the presentation to be normal by permuting $a_1,\ldots,a_k$ and using the operation
\begin{equation*}
\mathbb{Q}(\sqrt{a_i},\sqrt{a_j}) = \mathbb{Q}\left(\sqrt{a_i},\sqrt{\sqf(a_i a_j)}\right)
\end{equation*}
which we call multiplying $a_i$ onto $a_j$. The $i$-freeness of $K$ and $[K:\mathbb{Q}] = 2^k$ guarantee that all $a_i$ will be unequal to $\pm 1$ and hence have at least one prime divisor at any time during our transformation.\\
First permute the $a_i$'s such that $p_1\mid a_1$. Then multiply $a_1$ onto all other $a_i$ which are also divisible by $p_1$. Now $i_1=1$ and condition \eqref{eq:normal_b} holds for $j=1$. Also $i_1<\min(i_2,\ldots,i_k)$.\\
By induction we will show that for $m=1,\ldots, k$ we can transform the presentation such that condition \eqref{eq:normal_b} holds for all $j\leq m$ and $1=i_1<\ldots <i_m$ as well as $i_m<\min(i_{m+1},\ldots,i_k)$. For $m=1$ we did so above.\\
Let these hypotheses be satisfied for an $m<k$. Then for
\begin{equation*}
j_{m+1} := \min\{i : p_i\mid a_j, j > m\}
\end{equation*}
we know $j_{m+1}>i_m$. We permute the $a_{m+1},\ldots,a_k$ such that $p_{j_{m+1}}\mid a_{m+1}$. Hence $i_{m+1}=j_{m+1}$. After multiplying the new $a_{m+1}$ onto all other $a_i$ divisible by $p_{i_{m+1}}$ condition \eqref{eq:normal_b} is satisfied for $j=m+1$ and $i_{m+1}<\min(i_{m+2},\ldots,i_k)$. Hence the hypotheses now hold for $j=m+1$.\\
This result for $j=k$ proves the lemma.
\end{proof}

\begin{remark}
After transforming it to be normal, a presentation not necessarily satisfies condition \eqref{eq:mod4} any more.
\end{remark}

\begin{lemma}
Every $i$-free multi-quadratic number field has at most one normal presentation.
\end{lemma}
\begin{proof}
Let $(a_1,\ldots,a_n)$ and $(a_1',\ldots,a'_k)$ be two normal presentations of an $i$-free multi-quadratic number field $K$ and let $i_1,\ldots,i_k$ and $i_1',\ldots,i_k'$ be defined as in Definition \ref{def:normal}. Then $\sqrt{a_i'}\in K$ for all $i=1,\ldots ,k$ while
\begin{equation*}
B=\left\{\sqrt{\sqf\left( \textstyle \prod_{j\in S} a_j\right) }\,\middle|\, S\subseteq\{1,\ldots,k\}\right\}
\end{equation*}
is a generating set of $K$ seen as a $\mathbb{Q}$ vector space. Since the set $\{\sqrt{d}\mid d\in \mathbb{Z}, \text{square-free}\}$ is linear independent over $\mathbb{Q}$ (see Besicovitch \cite{besicovitch40}), a non-empty set $S_i\subseteq\{1,\ldots,k\}$ exists for every $i=1,\ldots,k$ such that $a_i'=\sqf(\prod_{j\in S_i} a_j)$.\\
For a $t\in\{1,\ldots,k\}$ let $m=\min(S_t)$. Let $i_1',\ldots,i_k'$ be defined for $a_1',\ldots,a_k'$  Then $i_m=i_t'$.
It is easy to see that $i_t'=\min_{m\in S_t} i_m$ and hence $i_t'\in\{i_1,\ldots,i_k\}$ holds for every $t=1,\ldots,k$. Since condition \eqref{eq:normal_a} holds for both sets of indices, we conclude $(i_1',\ldots,i_k')=(i_1,\ldots,i_k)$.\\
Let $l\in S_t$. Then $p_{i_l}\mid a_t'$ while according to condition \eqref{eq:normal_b} $a_l'$ is the only $a_i'$ divisible by $p_{i_l'}=p_{i_l}$. This means $l=t$ from which we conclude $S_t=\{t\}$ and $(a_1',\ldots,a_k')=(a_1,\ldots,a_k)$.
\end{proof}

For a square-free $P=p_1\ldots p_n$ let $R_k(P)$ and $Q_k(P)$ be the number of totally real and general multi-quadratic number fields respectively with $\rad(a_1\ldots a_n)=P$.\\
Since totally real number fields are $i$-free, we can use the lemmas above and count their normal forms in order to calculate $R_k(P)$. That means we are looking for all tuples $(a_1,\ldots, a_k)\in \mathbb{N}^k$ such that all conditions in Definition \ref{def:normal} are met. Knowing which of the primes $p_1,\ldots,p_n$ they are divisible by clearly determines $a_1,\ldots,a_k$ because they are square-free and positive. Therefore we must count the number of possible ways $p_1,\ldots,p_n$ can divide $a_1,\ldots,a_k$.\\
For every possible tuple $(i_1,\ldots,i_k)$, i.e every one with $1=i_1<i_2<\ldots<i_k$, we separately count the number of possibilities and then sum up.\\
For fixed $i_1,\ldots,i_k$ the choices which of the number $a_1,\ldots,a_k$ each prime $p_i$ divides are independent of each other. Hence we obtain the overall number of possibilities as the product of the possibilities for each $p_i$. Due to condition \eqref{eq:normal_b} the prime $p_{i_j}$ only divides $a_j$ for every $j=1,\ldots,k$. For a prime $p_i$ with $i_j<i<i_{j+1}$, $j=1,\ldots,k-1$ condition \eqref{eq:normal_a} states that $p_i\nmid a_l$ for all $l>j$. Hence $p_i$ can divide a non-empty subset of $a_1,\ldots, a_j$. That means for all $i_{j+1}-i_j-1$ of such $p_i$ there are $2^j-1$ possibilities. For the $n-i_k$ primes $p_i$ with $i>i_k$ there are $2^k-1$ possibilities since these may divide any of the number $a_1,\ldots,a_k$. Therefore
\begin{equation*}
R_k(P) = \sum_{i_2=i_1+1}^{n} \sum_{i_3=i_2+1}^{n} \ldots \sum_{i_k=i_{k-1}+1}^{n} 1^{i_2-i_1-1} 3^{i_3-i_2-1} \ldots (2^k-1)^{n-i_k}.
\end{equation*}
Using the identity
\begin{equation}\label{eq:ab}
\sum_{j=i+1}^n a^{j-i-1} b^{n-j} = \frac{1}{b-a} (b^{n-i}-a^{n-i})
\end{equation}
successively on all sums beginning with the inner one we can simplify the equation above.
To lighten the notation, we will write $\sum^*_j a_j := \sum_j c_j a_j$ for a finite sequence $a_j$, where $c_j$ are certain absolutely bounded positive real numbers (independent of P) that may change at each occurrence. Their exact values can easily be determined in every
instance, but play no role in the forthcoming discussion.
Furthermore let $\omega(P)$ be the number of distinct prime factors of $P$ and we obtain
\begin{lemma}\label{lem:R_k}
Let $P$ be square-free and $k\geq 1$. Then
\begin{equation*}
R_k(P) = F_k (2^k-1)^{\omega(P)-1}+\ssum_{j=1}^{k-1} (2^j-1)^{\omega(P)-1},
\end{equation*}
where
\begin{equation*}
F_k = \prod_{j=1}^{k-1}\frac{1}{2^k-2^j}.
\end{equation*}
\end{lemma}

Using this we calculate $Q_K(P)$. A general multi-quadratic number field may or may not be $i$-free. If it is not $i$-free, we can transform its presentation such that $a_1=-1$ and then multiply $a_1$ on all negatives among $a_2,\ldots, a_k$. This gives us a bijection to totally real multi-quadratic number fields $\mathbb{Q}(\sqrt{a_2},\ldots,\sqrt{a_k})$ with $\rad(a_2\ldots a_k)=P$. We have already calculated the number $R_{k-1}(P)$ of such number fields.\\
For the $i$-free number fields we proceed analogically to totally real number fields and count normal presentations. However, each of the numbers $a_1,\ldots,a_k$ may now be positive or negative. Hence there are $2^k R_k(P)$ possibilities. We conclude
\begin{equation*}
Q_k(P) = 2^k R_k(P) + R_{k-1}(P).
\end{equation*}
which together with Lemma \ref{lem:R_k} proves
\begin{lemma}\label{lem:Q_k}
Let $P$ be square-free and $k\geq 2$. Then
\begin{equation*}
Q_k(P) = 2^k F_k (2^k-1)^{\omega(P)-1}+\ssum_{j=1}^{k-1} (2^j-1)^{\omega(P)-1}.
\end{equation*}
\end{lemma}

Since the discriminant also depends on $r$ in Theorem \ref{theorem:discriminant}, we will now distinguish the number fields by the remainders of $a_1$ and $a_2$ modulo 4. We call a number field with $(a_1,a_2)\equiv (1,1)$ (mod 4) when its presentation satisfies \eqref{eq:mod4} a $(1,1)$-field. Let $Q_k^{(1,1)}(P)$ be the number of such fields where $\rad(a_1\ldots a_k)=P$. Analogously, we name the number of fields for the other cases in \eqref{eq:mod4} and do the same for totally real number fields.\\
If the radical is odd so are $a_1,\ldots,a_k$ and the number field is a $(1,1)$- or $(3,1)$-field. Similarly for an even radical one of $a_1,\ldots,a_k$ must be even and the number field is a $(2,1)$- or $(2,3)$-field. Hence for an odd, square-free $P$ we obtain
\begin{equation} \label{eq:sum_mod4}
\begin{split}
Q_k^{(1,1)}(P) + Q_k^{(3,1)}(P) &= Q_k(P)\\
Q_k^{(2,1)}(2P) + Q_k^{(2,3)}(2P) &= Q_k(2P).
\end{split}
\end{equation}

Since $(1,1)$-fields are $i$-free, we can again count normal presentations. After transforming a $(1,1)$-field to have a normal presentation, $a_i\equiv 1$ (mod 4) still holds for all $i=1,\ldots,k$. Similarly to totally real number fields $a_1,\ldots,a_k$ are clearly determined by their divisibilities by $p_1,\ldots,p_n$ since their sign is to be chosen such that $a_i\equiv 1$ (mod 4). That means $Q_k^{(1,1)}(P)=R_k(P)$.\\
As $(2,1)$-fields are also $i$-free we can proceed analogically. Transforming them to have a normal presentation using $p_1=2$ leads to $a_1\equiv 2$ and $a_i\equiv 1$ (mod 4)for $i=2,\ldots,n$. That means the sign is only determined for $a_2,\ldots,a_k$ while there are two possibilities for $a_1$. Hence $Q_k^{(2,1)}(2P)=2R_k(2P)$.\\
Together with \eqref{eq:sum_mod4} we have thereby proved
\begin{theorem} \label{theorem:Q_k_mod4}
Let $P$ be odd and square-free and $k\geq 2$. Then
\begin{align*}
Q_k^{(1,1)}(P) &= F_k (2^k-1)^{\omega(P)-1}+E_1\\
Q_k^{(3,1)}(P) &= (2^k-1) F_k (2^k-1)^{\omega(P)-1}+E_2\\
Q_k^{(2,1)}(2P) &= 2F_k (2^k-1)^{\omega(P)}+E_3\\
Q_k^{(2,3)}(2P) &= (2^k-2) F_k (2^k-1)^{\omega(P)}+E_4,
\end{align*}
where
\begin{equation*}
E_i = \ssum_{j=1}^{k-1} (2^j-1)^{\omega(P)-1}
\end{equation*}
for $1\leq i\leq 4$.
\end{theorem}
In order to get a similar result for totally real multi-quadratic number fields we use the following lemma which can easily be proved by induction.
\begin{lemma}\label{lem:poss}
For $a,b\in\mathbb{N}$ consider the number of possibilities of $a$-times selecting a non-empty subset of a $b$-element set where the parities of how often each element is to be selected overall are given. Should all elements be selected evenly often there are $((2^b-1)^a+(2^b-1) (-1)^a)/2^b$ possibilities, for all other combinations of parities there are $((2^b-1)^a - (-1)^a)/2^b$ possibilities.
\end{lemma}
When calculating $R_k^{(1,1)}(P)$ we must take into account that $a_1,\ldots,a_k$ must each be divisible by an even number of prime factors with $p_i\equiv 3$ (mod 4). Let the first $\omega_3=\omega_3(P)$ prime factors of $P$ be congruent 3 mod 4 and the remaining $\omega_1=\omega_1(P)$ be congruent 1 mod 4. Again we count the number of normal presentations for fixed $i_1,\ldots,i_k$. Here four cases can arise.\\
For $\omega_3<i_2$ exactly $\omega_3$ primes $p_i\equiv 3$ (mod 4) divide $a_1$ while none of them divide $a_2,\ldots,a_k$. Hence the number field is a $(1,1)$-field if and only if $\omega_3$ is even. Using \eqref{eq:ab} 
we calculate the number of possibilities in this case as
\begin{multline*}
\frac{1^{\omega_3}+(-1)^{\omega_3}}{2} \sum_{i_2=\omega_3+1}^{n} \sum_{i_3=i_2+1}^{n} \ldots \sum_{i_k=i_{k-1}+1}^{n} 1^{i_2-\omega_3-1} 3^{i_3-i_2-1} \ldots (2^k-1)^{n-i_k}\\
= (1^{\omega_3}+(-1)^{\omega_3}) \ssum_{j_1=1}^k (2^{j_1}-1)^{\omega_1}.
\end{multline*}
If $\omega_3=i_l$ for some $2\leq l\leq k$ then $a_l$ has exactly one prime factor $p_i\equiv 3$ (mod 4) which means $a_l\equiv 3$ (mod 4).\\
If $i_l<\omega_3<i_{l+1}$ for some $2\leq l\leq k$ then some of the primes $p_i$ with $1\leq i\leq i_l$ divide $a_1,\ldots,a_{l-1}$ and exactly one of them, namely $p_{i_l}$, divides $a_l$. That means the parity of the number of primes $p_i$ with $i_l<i\leq \omega_3$ which divide $a_j$ is fixed for every $j=1,\ldots,l$ in order to guarantee $a_j\equiv 1$ (mod 4). Since the number must be odd for $a_l$, there are $((2^l-1)^{\omega_3-i_l} - (-1)^{\omega_3-i_l})/2^l$ possibilities for the these primes according to Lemma \ref{lem:poss}. \\
When summing over all possibilities for $i_1,\ldots,i_k$ we can choose $i_1,\ldots,i_l$ and $i_{l+1},\ldots,i_k$ independently from each other because of $i_l<\omega_3<i_{l+1}$. Therefore the number of possibilities is the product of
\begin{multline*}
\sum_{i_2=i_1+1}^{\omega_3-1} \ldots \sum_{i_l=i_{l-1}+1}^{\omega_3-1} 1^{i_2-i_1-1} 3^{i_3-i_2-1}\ldots \frac{(2^l-1)^{\omega_3-i_l} - (-1)^{\omega_3-i_l}}{2^l}\\
= q_{-1}(-1)^{\omega_3-1} + \ssum_{j_3=1}^l (2^{j_3}-1)^{\omega_3-1}
\end{multline*}
with $q_{-1}\in\mathbb{Q}$ and
\begin{equation*}
\sum_{i_{l+1}=\omega_3+1}^n \ldots \sum_{i_k=i_{k-1}+1}^n (2^l-1)^{i_{l+1}-\omega_3-1} \ldots (2^k-1)^{\omega(P)-i_k} = \ssum_{j_1=l}^k (2^{j_1}-1)^{\omega_1}.
\end{equation*}
Both factors were transformed using \eqref{eq:ab}.\\
If $i_k<\omega_3$, similarly to the previous case, there are $((2^k-1)^{\omega_3-i_k} - (-1)^{\omega_3-i_k})/2^k$ possibilities for the primes $p_i$ with $i_k<i\leq \omega_3$. Therefore there are
\begin{multline*}\label{eq:k_ung_3}
\sum_{i_2=1_1+1}^{\omega_3-1} \ldots \sum_{i_k=i_{k-1}+1}^{\omega_3-1} 1^{i_2-i_1-1} \ldots \frac{(2^k-1)^{\omega_3-i_k} - (-1)^{\omega_3-i_k}}{2^k} (2^k-1)^{n-\omega_3}\\
= \frac{F_k}{2^k} (2^k-1)^{n-1} + \left( q_{-1}(-1)^{\omega_3-1} + \ssum_{j_3=1}^l (2^{j_3}-1)^{\omega_3-1}\right)  (2^k-1)^{\omega_1}
\end{multline*}
possibilities in this case. Adding up all cases gives us
\begin{multline*}
R_k^{(1,1)}(P) = \frac{F_k}{2^k} (2^k-1)^{n-1}\\
+ \ssum_{\substack{1\leq j_3\leq k-1\\ j_3\leq j_1\leq k}} (2^{j_3}-1)^{\omega_3-1}(2^{j_1}-1)^{\omega_1} + \ssum_{1\leq j_1\leq k}(-1)^{\omega_3-1}(2^{j_1}-1)^{\omega_1}.
\end{multline*} 
In a similar manner one can compute $R_k^{(2,1)}(P)$. Since \eqref{eq:sum_mod4} also holds for totally real number fields we obtain
\begin{theorem}\label{theorem:R_k_mod4}
Let $P$ be odd and square-free and $k\geq 2$. Then
\begin{align*}
R_k^{(1,1)}(P) &= \frac{1}{2^k} F_k (2^k-1)^{\omega(P)-1}+E_1\\
R_k^{(3,1)}(P) &= \frac{2^k-1}{2^k} F_k (2^k-1)^{\omega(P)-1}+E_2\\
R_k^{(2,1)}(2P) &= \frac{1}{2^{k-1}} F_k (2^k-1)^{\omega(P)}+E_3\\
R_k^{(2,3)}(2P) &= \frac{2^{k-1}-1}{2^{k-1}} F_k (2^k-1)^{\omega(P)}+E_4,
\end{align*}
where
\begin{equation*}
E_i = \ssum_{\substack{1\leq j_3\leq k-1\\ j_3\leq j_1\leq k}} (2^{j_3}-1)^{\omega_3(P)-1}(2^{j_1}-1)^{\omega_1(P)} + \ssum_{1\leq j_1\leq k}(-1)^{\omega_3(P)-1}(2^{j_1}-1)^{\omega_1(P)}
\end{equation*}
for $1\leq i\leq 4$.
\end{theorem}

Using theorem \ref{theorem:discriminant} we can express $N_k(X)$ as a sum of terms from Theorem \ref{theorem:Q_k_mod4}.
\begin{theorem}\label{theorem:N_k_sum}
\begin{equation*}
\begin{split}
N_k(X)=&\sum_{\substack{\mu^2(P)=1,\, 2\nmid P\\ P^{2^{k-1}}\leq X}} Q_k^{(1,1)}(P) + \sum_{\substack{\mu^2(P)=1,\, 2\nmid P\\ (2^2P)^{2^{k-1}}\leq X}} Q_k^{(3,1)}(P)\\
+ &\sum_{\substack{\mu^2(P)=1,\, 2\nmid P\\ (2^2 2P)^{2^{k-1}}\leq X}} Q_k^{(2,1)}(2P) + \sum_{\substack{\mu^2(P)=1,\, 2\nmid P\\ (2^3 2P)^{2^{k-1}}\leq X}} Q_k^{(2,3)}(2P)
\end{split}
\end{equation*}
\end{theorem}
In the same way $N_k^+(X)$ is calculated from the terms from Theorem \ref{theorem:R_k_mod4}.

\section{Asymptotic analysis}
In order to calculate the sum in Theorem \ref{theorem:N_k_sum} asymptotically, we need to calculate sums of the form
\begin{equation*}
A_M(x):=\sum_{\substack{n\leq x\\ \mu^2(n)=1,\, 2\nmid n}}M^{\omega(n)}
\end{equation*}
where $M\in\mathbb{N}$. For that we use Perrons formula in form of the following corollary from it (see Brüdern \cite{bruedern95}, p.28).
\begin{lemma}\label{lem:perron2}
Let $c>0$, $x\geq 2$, $T\geq 2$. Let $a:\mathbb{N}\to\mathbb{C}$ be an arithmetic function whose Dirichlet series is absolutely convergent for $s=c$. Then
\begin{multline*}
\sum_{n\leq x}a(n) = \frac{1}{2\pi i}\int_{c-iT}^{c+iT}\left(\sum_{n=1}^{\infty}a(n)n^{-s}\right)\frac{x^s}{s}ds\\
+ \mathcal{O}\left(\frac{x^c}{T}\sum_{n=1}^{\infty}|a(n)|n^{-c} + A_x\left(1+\frac{x\log x}{T}\right)\right)
\end{multline*}
where
\begin{equation*}
A_x=\max_{\frac{3}{4}x\leq n\leq \frac{5}{4}x} |a(n)|.
\end{equation*}
\end{lemma}
Since $a(n)=\mathbbmss{1}_{2\nmid n} \mu^2(n) M^{\omega(n)}$ is multiplicative the corresponding Dirichlet series is
\begin{equation}\label{eq:dirichlet_series}
\sum_{2\nmid n}\frac{\mu^2(n) M^{\omega(n)}}{n^s} = \prod_{p\neq 2} \left(1+\frac{M}{p^s}\right) = H(s) \zeta(s)^M,
\end{equation}
where
\begin{equation}\label{eq:H(s)_def}
\begin{split}
H(s) &:= \zeta(s)^{-M} \prod_{p\neq 2}\left(1+\frac{M}{p^s}\right)\\
&= \left(1-\frac{1}{2^s}\right)^M \prod_{p\neq 2} \left(1+\frac{\alpha_2}{p^{2s}}+\ldots+\frac{\alpha_{M+1}}{p^{(M+1)s}}\right)
\end{split}
\end{equation}
with $\alpha_2,\ldots,\alpha_{M+1}\in\mathbb{Z}$. It is easy to see that $H(s)$ is holomorphic for $\re s>\frac{1}{2}$. Since $H(s)$ and $\zeta(s)$ converge for $s=1+\epsilon$ the Dirichlet series is absolutely convergent there and we can apply Lemma \ref{lem:perron2} with $c=1+\epsilon$ for any $\epsilon>0$.\\
Then, since $A_x\ll x^{\epsilon'}$ for any $\epsilon'>0$, the error term is bounded from above by $x^{1+\epsilon}/T+x^\epsilon$. In order to estimate the integral from above, we use the fact that there is a $\frac{1}{2}<\delta<1$ for every $M\in\mathbb{N}$ such that
\begin{equation}\label{eq:zeta_1}
\frac{1}{T}\int_{1}^T|\zeta(\delta +it)|^{M} dt = \mathcal{O}(1).
\end{equation}
This follows from Titchmarsh \cite{titchmarsh86}, p.151. Since $\zeta(s)$ has a pole of order 1 with residue 1 in $s=1$ the function $f(s):=H(s)\zeta(s)^Mx^s/s$ has a pole of order $M$ in $s=1$ and one can easily verify
\begin{equation*}
\underset{s=1}{\res} f = \frac{H(1)}{(M-1)!}xP_M'(\log x)
\end{equation*}
where $P_M'$ is a monic polynomial of degree $M-1$. Using the residue theorem we write the integral from Lemma \ref{lem:perron2} as
\begin{equation}\label{eq:integral}
\int_{c-iT}^{c+iT} f(s)ds = 2\pi i \,\underset{s=1}{\res} f + \int_{\delta-iT}^{\delta+iT} f(s)ds + \int_{\delta+iT}^{c+iT} f(s)ds + \int_{c-iT}^{\delta-iT} f(s)ds.
\end{equation}
As $H(s)$ is bounded for $\delta\leq \re s\leq 1+\epsilon$ we can estimate the first integral by
\begin{equation*}
x^\delta\left(1+\int_1^T |\zeta(\delta+it)^M|\frac{dt}{t}\right). 
\end{equation*}
From \eqref{eq:zeta_1} we know that $\int_{T/2}^T |\zeta(\delta +it)|^M \frac{dt}{t}=\mathcal{O}(1)$. By successively replacing $T$ by $T/2, T/4,\ldots$ and adding up we can finally estimate by $x^\delta (1+\log T)$.\\
Using the bound $\zeta(\sigma+it)\ll |t|^{(1-\omega)/2+\epsilon/2}$ for $0\leq \sigma\leq 1+\epsilon$ (see Brüdern \cite{bruedern95}) we can estimate the second and third integral in \eqref{eq:integral} by
\begin{equation*}
\frac{1}{T}\int_{\delta}^c |\zeta(\sigma +iT)|^M x^\sigma d\sigma
\ll \frac{T^{(1+\epsilon)M/2}}{T} \int_\delta^c \left(\frac{x}{T^{M/2}}\right)^\sigma d\sigma 
\ll \frac{x^{1+\epsilon}}{T}
\end{equation*}
where the last estimation holds for $x>T^{M/2}$. Finally, by choosing $T=x^{1/M}$ and $\epsilon$ small enough we see that all integrals in \eqref{eq:integral} as well as the error term in Lemma \ref{lem:perron2} are bounded by $\mathcal{O}(x^d)$. Hence we get
\begin{equation*}
A_M(x) = \frac{H(1)}{(M-1)!}x P_M'(\log x) + \mathcal{O}(x^d).
\end{equation*}
Using this and Theorem \ref{theorem:Q_k_mod4} we get
\begin{multline*}
\sum_{\substack{\mu^2(P)=1,\, 2\nmid P\\ P^{2^{k-1}}\leq X}} Q_k^{(1,1)}(P) = \frac{F_k}{2^k-1} A_{(2^k-1)}\left( x^{1/2^{k-1}} \right) + \ssum_{j=1}^{k-1} A_{(2^j-1)}\left( x^{1/2^{k-1}}\right) \\
= \frac{F_k}{2^k-1} \frac{H(1)}{(2^k-2)!} x^{1/2^{k-1}} \left(\frac{1}{2^{k-1}}\right)^{2^k-2}\widetilde{P}_k(\log x) + \mathcal{O}\left(x^{d/2^{k-1}}\right) 
\end{multline*}
where $\widetilde{P}_k$ is a monic polynomial of degree $2^k-2$. Accordingly we derive formulas for the other sums in Theorem \ref{theorem:N_k_sum} which differ in the leading factor. By adding all up we conclude
\begin{equation*}
N_k(x) = C_k x^{1/2^{k-1}} P_k(\log x) + \mathcal{O}\left(x^{d/2^{k-1}}\right),
\end{equation*}
where $P_k$ is again a monic polynomial of degree $2^k-2$ and
\begin{equation*}
C_k = \left(\frac{1}{2^k-1}+\frac{1}{4} + \frac{1}{4} +\frac{2^k-2}{16}\right) F_k \frac{H(1)}{(2^k-2)!} \left(\frac{1}{2^{k-1}}\right)^{2^k-2}.
\end{equation*}
This proves part of Theorem \ref{theorem:main_result}. Calculating $N_k^+(x)$ works in the same way. However, this time we also need to sum up terms from Theorem \ref{theorem:R_k_mod4} of the form $M^{\omega_1(n)} N^{\omega_3(n)}$ for odd $M\in\mathbb{N}$ and $N\in\mathbb{N}\cup \{-1\}$ with $M\geq N$. The Dirichlet series for these is
\begin{equation*}
\sum_{2\nmid n}\frac{\mu^2(n) M^{\omega_1(n)}N^{\omega_3(n)}}{n^s} 
= H^+(s) L(s,\chi)^{(M-N)/2}\zeta(s)^{(M+N)/2}.
\end{equation*}
All estimates work just as for $A_M(x)$. However, since $L(s,\chi)$ is holomorphic in $s=1$ we this time integrate over a function with a pole of degree $(M+N)/2$. That means $A_{M,N}(x)$ will not contribute to the leading term of the final result because $M\leq 2^{k-1}-1, N\leq 2^k-1$ in Theorem \ref{theorem:R_k_mod4}. This furnishes the claim for $N_k^+(x)$ and completes the proof of Theorem \ref{theorem:main_result}.

\end{document}